\documentclass[leqno, letterpaper, 11pt]{article}

\usepackage{amssymb}
\usepackage{amsmath}
\usepackage{amsthm}
\usepackage{fullpage}
\usepackage{mathrsfs}
\usepackage{subfigure}
\usepackage{verbatim}
\usepackage{enumitem}
\usepackage{bbm}
\usepackage{tikz}
\usetikzlibrary{through,calc,positioning,decorations.pathreplacing}

\renewcommand{\P}{\mathbb{P}}

\newcommand{\bE}{{\mathbb E}}
\newcommand{\bR}{{\mathbb R}}
\newcommand{\bZ}{{\mathbb Z}}

\newcommand{\cC}{{\mathcal C}}

\newcommand{\cG}{{\mathcal G}}

\newcommand{\cR}{{\mathcal R}}

\newcommand{\cB}{\mathcal B}

\newcommand{\cI}{\mathcal I}

\newcommand{\balpha}{\overline\alpha}

\newcommand{\red}{\widetilde{R}}
\newcommand{\blue}{\widetilde{B}}
\newcommand{\Len}{\mathrm{Len}}
     
\numberwithin{equation}{section}
     
\definecolor{ashgrey}{rgb}{0.7, 0.75, 0.71}
\definecolor{faint}{rgb}{0.9, 0.9, 0.9}
\definecolor{lgrey}{rgb}{0.8, 0.8, 0.8}
\definecolor{dgrey}{rgb}{0.5, 0.5, 0.5}

\newtheorem{theorem}{Theorem}[section]

\newtheorem{lemma}[theorem]{Lemma}
\newtheorem{corollary}[theorem]{Corollary}

\newtheorem{open}[theorem]{Open problem}

\newcommand{\prob}[1]{\mathbb{P}\left(#1\right)}
\newcommand{\probsub}[2]{\mathbb{P}_{#1}\left(#2\right)}

\begin{document}

\begin{center}\Large
{\bf {Competing deterministic growth models in two dimensions}}
\end{center}

\begin{center}
{\sc Janko Gravner}\\
{\rm Department of Mathematics}\\
{\rm University of California}\\
{\rm Davis, CA 95616}\\
{\rm \tt gravner{@}math.ucdavis.edu}
\end{center}

\begin{center}
{\sc David Sivakoff}\\
{\rm Departments of Statistics and Mathematics}\\
{\rm The Ohio State University}\\
{\rm Columbus, OH 43210, USA}\\
{\rm \tt dsivakoff{@}stat.osu.edu}
\end{center} 




\begin{abstract} 
We consider three-state cellular automata in two dimensions in 
which two colored states, blue and red, compete for control 
of the empty background, starting from low initial densities $p$ and $q$. When the dynamics of both colored types
are one-dimensional, the dynamics has three distinct phases, characterized
by a power relationship between $p$ and $q$: 
two in which one of the colors is prevalent, and one when the colored 
types block 
each other and leave most of the space forever empty. When 
one of the colors spread in two dimensions and the other in 
one dimension, we also establish a power relation between $p$ and $q$ 
that characterizes
which of the two colors eventually controls most of the space. 
\end{abstract}

\section{Introduction}\label{sec:intro}
Competing growth models were introduced in the probabilistic literature 
in the 1990s as a class of deterministic cellular automata 
called threshold voter models \cite{GG}, and as multiple-type first passage percolation dynamics \cite{HP1,HP2}. The latter processes have been explored further in \cite{Hof, SS, CS, FS1,FS2}. In these models, two types of particles compete for space using simple infection rules, which may not be the same for both types; for example, one of the types may spread faster or require ``activation.'' Such models are intrinsically non-monotone, which poses a challenge in rigorous analysis. Another competition between two types of particles, introduced in \cite{GM}, is when 
one of the types is inert and simply provides a random environment --- a field of obstacles --- in which the other type tries to spread. See 
\cite{GH, GHS, Gho} for recent progress on such dynamics. In the 
context of the present work, these obstacles undergo degenerate, zero-dimensional, growth. By contrast, \cite{GG} studies cases
when the types both spread in two dimensions (with sufficient 
regularity). It is thus natural to consider the remaining possibility 
whereby the 
two growths have different non-zero dimensions and we
investigate such competing growth models on the planar
lattice; the two types may
in addition have different range and spread at different speeds. 
Our motivation comes from the earlier work on models that combine one- and 
two-dimensional growth features \cite{BGSW}. Recent physical literature also includes experiments and mathematical 
models of crystal growth that may vary the dimensionality depending 
on nucleation conditions \cite{YLG+} and in particular 
features competition between one-dimensional crystals growing 
in different directions \cite{CLY+, YLG+}, which warrants further mathematical investigation of such dynamics.

In our model, each vertex of $\bZ^2$ is independently initially 
\begin{equation}\label{eq:initial state}
\begin{cases}
\text{blue} & \text{with probability } p\\
\text{red} & \text{with probability } q\\
\text{empty} & \text{with probability } 1-p-q.
\end{cases}
\end{equation}
Let $\cB\subset \bZ\times \{0\}$ and $\cR\subset \bZ^2$ be finite subsets of the lattice,
and fix an update interval for red  $r >0$. At times $t\in\{1, 2, 3, \ldots\}$, every empty vertex $x\in \bZ^2$ that has a blue neighbor in $x+\cB$ becomes blue. At times $t\in \{r, 2r,3r, \ldots\}$, every empty vertex $x\in \bZ^2$ that has a red neighbor in $x+\cR$ becomes red. Ties are decided according to independent fair coin flips. In this way, our model is a competition between two deterministic first-passage processes, where blue spreads via its neighborhood at speed $1$, and red spreads via its neighborhood at speed $1/r$.

We denote the resulting configuration at time $t\in [0,\infty)$ by $\xi_t$. As 
every site in $\bZ^2$ changes state at most once, we may define the {\em final configuration} $\xi_\infty$, in which every site has its final state. Properties of this random configuration are of 
our primary interest. 


Our first result is on competing one-dimensional processes, resembling the 
dynamics in \cite{CLY+}, in which red grows vertically with range $\rho$ and blue grows horizontally with range $\tau$. 


\begin{figure}
    \centering
    \includegraphics[width=0.45\linewidth]{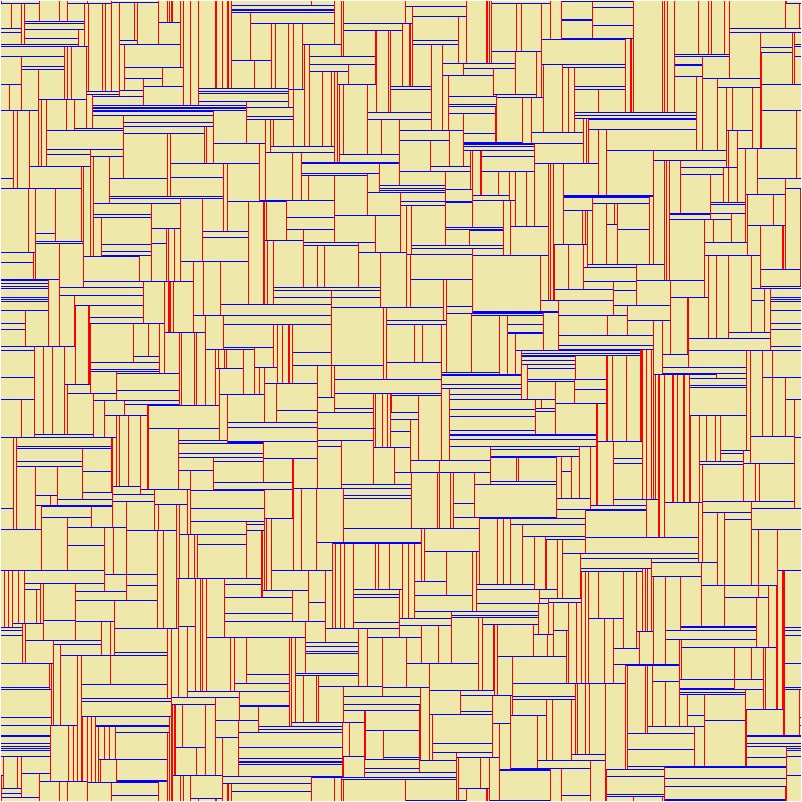}\hskip0.3cm
    \includegraphics[width=0.45\linewidth]{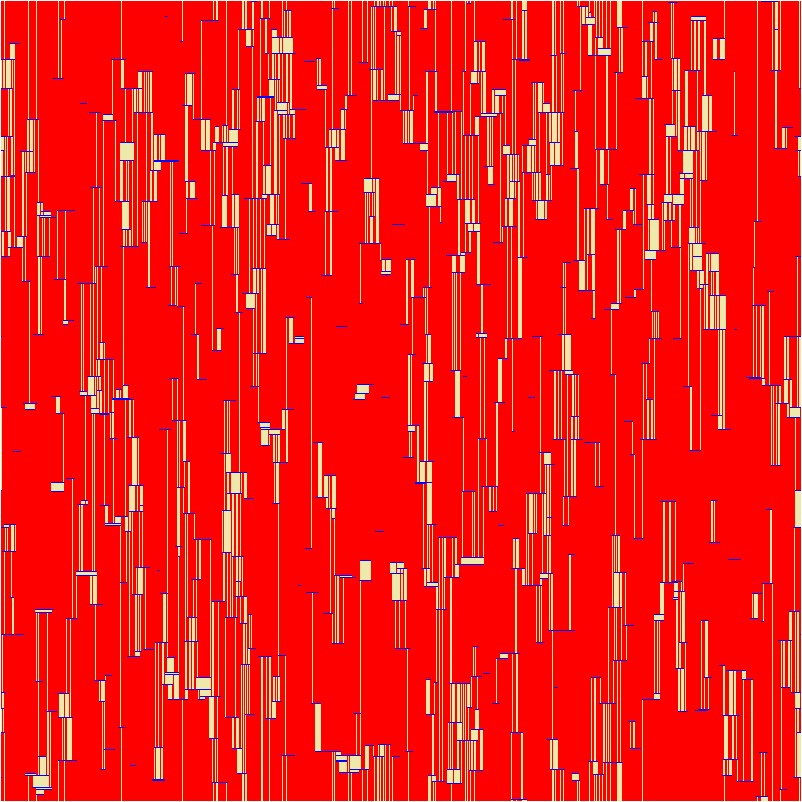}       
    \caption{Illustration of Corollary~\ref{thm:1d three phases} when $\rho=\tau=1$: final configurations on an $800\times 800$ torus when $p=0.001$, $q=0.001$ (left) and $q=0.02$ (right).}
    \label{fig:1vs1}
\end{figure}

\begin{theorem}\label{thm:1d long-range} Assume that
$\cR = \{0\}\times ([-\rho,\rho]\cap \bZ)$ and $\cB = ([-\tau,\tau]\cap \bZ)\times \{0\}$, with $\rho,\tau \in \bZ_{\ge1}$. If $q = a p^{\gamma}$ where $\gamma = \frac{1}{\tau} + \frac{\rho}{\rho+1}$, then
   $$
0<\liminf_{p\to 0} \P(0 \text{ is eventually blue}) \le \limsup_{p\to 0} \P(0 \text{ is eventually blue}) < 1.
$$
Moreover, the $\limsup$ tends to $0$ as $a \to \infty$ and the $\liminf$ tends to $1$ as $a \to 0$.
\end{theorem}

We conclude that, as $q$ increases compared to $p$, 
the dynamics can be in three distinct phases, dominated by blue, neither, 
or red. For example, when $\rho=\tau=1$, the three phases occur when $q\ll p^{3/2}$, 
$p^{3/2}\ll q\ll p^{2/3}$ and $q\gg p^{2/3}$. See Figure~\ref{fig:1vs1}
for an illustration.

\begin{corollary}\label{thm:1d three phases} Assume the same setting 
as in the above Theorem~\ref{thm:1d long-range}, and additionally
denote $\gamma'=\frac{1}{\rho} + \frac{\tau}{\tau+1}$. Then $\gamma\gamma'>1$ and
\begin{itemize}
    \item if $q\ll p^\gamma$, $\P(0 \text{ is eventually blue})\to 1$;
    \item if $p^\gamma\ll q\ll p^{1/\gamma'}$, $\P(0 \text{ remains empty for all time})\to 1$; and 
    \item if $p^{1/\gamma'}\ll q$, $\P(0 \text{ is eventually red})\to 1$.
\end{itemize}
\end{corollary}

Our second result addresses competition between a two-dimensional 
red growth with range $\rho$ and a one-dimensional blue one. To 
strengthen the second part of the theorem, we show that the blue type 
still wins in the appropriate regime even if it 
only spreads leftward,
with the smallest possible range $1$. See Figure~\ref{fig:2vs1}
for an illustration. 

\begin{figure}
    \centering
    \includegraphics[width=0.45\linewidth]{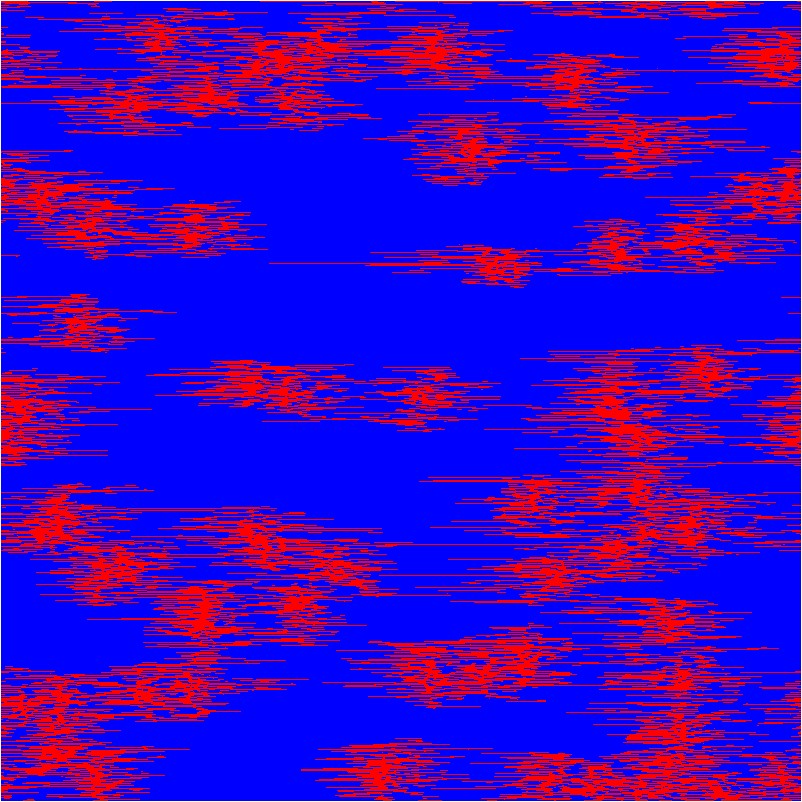}\hskip0.3cm
    \includegraphics[width=0.45\linewidth]{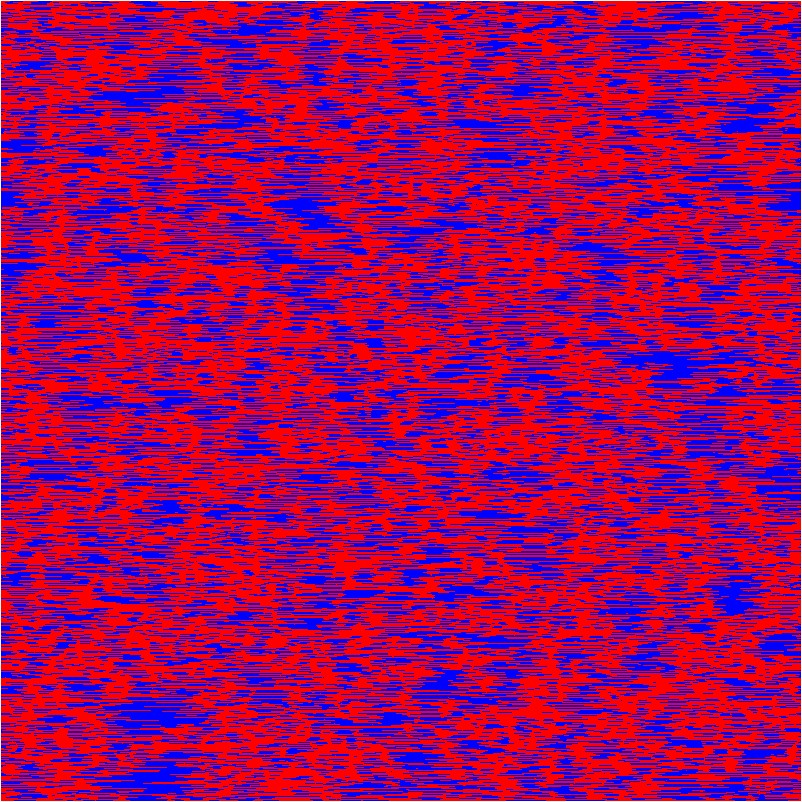}       
        \caption{Illustration of Theorem~\ref{thm:1vrho} when $\rho=2$ and $\cB=\{e_1\}$: final configurations on $800\times 800$ torus when $p=0.04$, $q=0.0001$ (left) and $q=0.001$ (right).}
    \label{fig:2vs1}
\end{figure}
 
\begin{theorem}\label{thm:1vrho} (1) Assume that 
$\cR=\{x\in\bZ^2: ||x||_1\le \rho\}$ for some $\rho\in \bZ_{\ge1}$. 
If $e_1\in \cB$ and $p,q$ are not both zero, then a.s.~every site in $\bZ^2$ 
is eventually nonempty. 

In (2) and (3) below, assume that $q=a p^{1+\rho/(\rho+1)}$.

\noindent (2) Suppose first that $\cB = ([-\tau,\tau]\cap \bZ)\times \{0\}$, with $\tau \in \bZ_{\ge1}$. Then for every $a>0$,
$$
\limsup_{p\to 0} \P(0 \text{ is eventually blue}) < 1
$$
and tends to $0$ as $a\to\infty$. 

\noindent (3) Suppose next that $\cB=\{e_1\}$. Then for every $a>0$,
$$
\liminf_{p\to 0} \P(0 \text{ is eventually blue})>0
$$
and tends to $1$ as $a \to 0$. 
\end{theorem}

We give the very simple proof of the first statement 
right away. 

\begin{proof}[Proof of (1)]
Let 
$A=\inf\{a\ge 0: (a,0)\text{ is nonempty in the final configuration}\}$. 
As at least one of $p$ and $q$ is nonzero, $P(A<\infty)=1$. If $a>0$ is finite, 
then on the event $\{A=a\}$ there is 
a finite time $t$ at which $(a,0)$ is colored, 
which causes $(a-1,0)$ to be 
colored by time $t+(1\vee r)$. 
This implies that 
$\{A<\infty\}=\{A=0\}$.
\end{proof}

The proof of part (2) of Theorem~\ref{thm:1vrho} is relatively straightforward, and the key lemma is proved in Section~\ref{sec:red wins}. The most substantial effort is devoted to proving 
part (3) of
Theorem~\ref{thm:1vrho}. This involves a multiscale argument, whose 
essential ingredients are already necessary for the nearest 
neighbor case $\rho=1$. To make the presentation more readable, 
we therefore assume $\rho=1$ until Section~\ref{sec:long range}. The bulk of the argument is in Section~\ref{sec:blue wins}, and in Section~\ref{sec:long range}  we explain the modifications necessary for the general case. We also complete the proofs of Theorem~\ref{thm:1d long-range} and Corollary~\ref{thm:1d three phases} in Section~\ref{sec:long range}. In Section~\ref{sec:open} we conclude with a discussion of open problems and extensions to multicolor dynamics.

\section{Red wins}\label{sec:red wins}
In this section, we prove the following lemma, which implies part (2) of Theorem~\ref{thm:1vrho}. 
\begin{lemma}\label{lem:red wins}
Assume that $\cR=\{\pm e_1, \pm e_2\}$, 
$r>0$ is arbitrary, and $\cB = ([-\tau,\tau]\cap \bZ)\times \{0\}$, with $\tau \in \bZ_{\ge1}$.
    If $q = ap^{3/2}$, then
    $$
 \liminf_{p\to 0} \P(0 \text{ is eventually red}) >0
    $$
    for every $a>0$ and tends to $1$ as $a\to\infty$.
\end{lemma}
\begin{proof}
    To gain independence between red and blue vertices, we construct the initial configuration in two stages. First, mark each site independently as red with probability $q$. Next, mark each site independently potentially blue with probability $p' = \frac{p}{1-q}$. 
    To realize the initial state (\ref{eq:initial state}), label each site marked red as red, label each site that is not red but potentially blue as blue, and label each remaining site empty.

\begin{figure}
\begin{center}
\begin{tikzpicture}


\draw[color=black, thick] (-5,0)--(5,0);
\draw[color=black, thick] (0,0)--(0,0.2);
\fill[lgrey] (0,0) rectangle (5,1);
\draw [fill=lgrey,ultra thick] (3.6-1,0) rectangle (3.6+1,1);
\node [left] at (0,0.5) {$A+1$};
\draw [line width=0.5, <-] (-0.55, 0) -- (-0.55,0.25);
\draw [line width=0.5, ->] (-0.55, 0.75) -- (-0.55,1);
\node [above] at (3.6,1) {$2A+1$};
\draw [line width=0.5, <-] (3.6-1, 1.3) -- (3.6-0.7,1.3);
\draw [line width=0.5, ->] (3.6+0.7, 1.3) -- (3.6+1,1.3);

\draw[color=black, ultra thick] (-6,0)--(6,0);
\draw[color=black, thick] (0,0)--(0,0.15);
\draw[color=black, thick] (5,0)--(5,0.15);
\draw[color=black, thick] (-5,0)--(-5,0.15);
\draw[color=black, thick] (6,0)--(6,0.15);
\draw[color=black, thick] (-6,0)--(-6,0.15);

\fill [red] (3.6, 0.6) circle (2pt);

\node [below] at (0,0) {$(0,0)$};
\node [below] at (5,0) {$(B,0)$};
\node [below] at (6.7,0) {$(B+A,0)$};
\node [below] at (-5,0) {$(-B,0)$};
\node [below] at (-7,0) {$(-(B+A),0)$};
  
\end{tikzpicture}
\end{center}
\caption{Assuming $r=\tau=1$, the origin becomes red if there are
integers $A$ and $B$ so that, initially, there is a red site in the rectangle $[0,B]\times [0,A]$ (grey), no blue site in the outlined $(2A+1)\times (A+1)$ rectangle centered horizontally around that red site, and also no blue site 
in the interval on  the $x$ axis from $(-(B+A),0)$ to $(B+A,0)$.} 
\label{fig:red wins}
\end{figure}
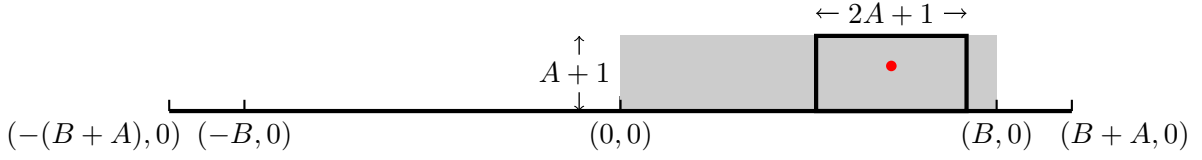

    We define the following two events for a constant $\epsilon>0$ and small enough $p$:
    \begin{equation}\label{eq:GH}
    \begin{aligned}
        G &= \{\text{there is no potentially blue site in } [-\epsilon \tau r/p,\epsilon \tau r/p]\times\{0\}\}, \quad \text{and} \\
        H &= \{\text{there is a red site $(x,y)\in [0,\epsilon/2p]\times [0,\epsilon/\sqrt{p}]$ and}\\
        &\qquad \text{ no potentially blue site in $[x-2\epsilon \tau r/\sqrt{p}, x+2\epsilon \tau r/\sqrt{p}]\times [0,\epsilon/\sqrt{p}]$}\}.
    \end{aligned}
    \end{equation}
    Observe that on the event $G\cap H$, the origin becomes red because the red site at $(x,y)$ causes $(x,0)$ to become red by time $\epsilon \tau r / \sqrt{p}$, before blue can interfere, and this red growth propagates to the origin in additional time at most $\epsilon \tau r / 2p$, before the closest blue could reach the origin. See 
    Figure~\ref{fig:red wins}. 

    Fix $\delta>0$. Choose $\epsilon$ small enough that $3\epsilon \tau r <\delta$ and $6 \epsilon^2 \tau r <\delta$, so that for small enough $p$ (and $q$), we have $\P(G^c)<\delta$ and
    \begin{equation*}
    \begin{aligned}
    &\P(\text{there exists a potentially blue site in $[-2\epsilon \tau r /\sqrt{p},2\epsilon \tau r /\sqrt{p}]\times [0,\epsilon/\sqrt{p}]$}) \\&\le \frac{5\epsilon^2 \tau r}{p} p' \le 6\epsilon^2 \tau r < \delta.
    \end{aligned}
    \end{equation*}
    Then,
    $$
    \P(\text{there is no red site in } [0,\epsilon/2p]\times [0,\epsilon/\sqrt{p}]) \le (1-q)^{\epsilon^2 / 3p^{3/2}} \le \exp(- \epsilon^2 a / 3).
    $$
     It follows by independence between red sites and potentially blue sites that
    $$
    \P(H)\ge (1-\delta)(1-\exp(-\epsilon^2 a/3)).
    $$
    Since $G$ and $H$ are positively correlated, for every $a>0$ and all sufficiently small $p$ we have
    $$
    \P(0 \text{ is eventually red}) \ge \P(G\cap H)\ge (1-\delta)^2(1-\exp(-\epsilon^2 a / 3) )>0.
    $$
    This proves the first claim of the lemma. The second claim follows by taking $a$ sufficiently large such that $\exp(-\epsilon^2 a/3)<\delta$.
\end{proof}

\section{Blue wins}\label{sec:blue wins}
We prove in this section the lower bound on the probability that the origin is eventually blue when $\cB = \{e_1\}$ and $\cR = \{\pm e_1,\pm e_2\}$. Observe that we have monotonicity in the red configurations, meaning that adding red sites at any time results in more sites being colored red and fewer sites colored blue at later times. Therefore, without loss of generality, we can assume that $r\le 1$. 

For blue to win, it is sufficient that the following conditions hold:
\begin{enumerate}
    \item there is an initially blue vertex in $[0, C/p]\times \{0\}$; and
    \item no vertex in $[-C/p, C/p]\times \{0\}$ is colored red by time $C/p$.
\end{enumerate}
Condition 1 is easy to guarantee by choosing $C$ large. Condition 2 is much more subtle. We will construct a blocking configuration of initially blue sites, whose eventual growth will intercept any red sites below its center, thus preventing red from reaching the $x$-axis from below. By symmetry, a similar structure will also stop red from above.

\subsection{A continuous comparison process}
We describe here a process in continuous time and space, for which we show that blue can slow the vertical progression of red. In this process, we let $\red_t\subset \bR^2$ denote the set of red points at time $t\in [0,\infty)$, and let $\blue\subset\bR^2$ be the set of blue points, which will not change in time in the comparison process. Red points will grow vertically at infinite speed, and horizontally at speed $1/\alpha$ with $\alpha\in (0, \infty)$, but cannot pass through blue sites, and must grow around them. More precisely, for $z,z'\in \bR^2$, let $\Gamma(z,z')$ denote the set of paths in $\bR^2$ that start at $z$, end at $z'$, and move only horizontally or vertically. That is, each path in $\Gamma(z,z')$ consists of line segments connecting consecutive points in $z=z_0, z_1, \ldots, z_n = z'$ for some $n$ and such that $z_{i} - z_{i-1}$ has exactly one coordinate equal to $0$. The length of a path $\gamma = (z_0, \ldots, z_n)\in \Gamma(z,z')$ is $\Len(\gamma) := \alpha \sum_{i=1}^n |x_i - x_{i-1}|$, where $z_i=(x_i,y_i)$, so $\Len(\gamma)$ is $\alpha$ times the total horizontal displacement along $\gamma$. Then, given the sets $\red_0$ and $\blue$, the set of red points at time $t$ is
$$
\red_t = \{z\in \bR^2 : \inf\{ \Len(\gamma): \gamma \in \Gamma(z,z'), z'\in \red_0, \gamma\cap \blue = \emptyset\} \le t\}.
$$

We now describe the initial sets $\red_0$ and $\blue$. Since red grows vertically at infinite speed, the vertical displacements between points in these sets will be immaterial, so we choose unit spacing. For $f,g,h\in (0,\infty)$ and integer $m\ge 1$, let
$$
\red_0 = [0,f]\times \{0\},
$$
and
$$
\blue = \bigcup_{k=1}^m \Big\{[(k-1)g-kh,kg-kh] \times \{k\}\Big\}.
$$
Note that $\blue$ consists of $m$ horizontal intervals of length $g$.
We want the horizontal extent of $\blue$ to exceed that of $\red_0$ by $h$ on either side, so we require
\begin{equation}\label{fgh_constraint}
m(g-h) = f+h.
\end{equation}
With this choice, observe that $\red_t$ contains points with arbitrarily large (positive) second coordinate only after time $\alpha h$. This follows because the horizontal overlap between successive blue intervals is $h$, and the overhang beyond $[0,f]$ is also $h$. So red ``escapes'' blue after this time, and the horizontal extent of $\red_{\alpha h}$ is $[-h, f+h]$.

\begin{figure}[ht!]
\begin{center}
\begin{tikzpicture}



\def\a{1}
\def\f{4}
\def\m{3}

\pgfmathsetmacro\l{{(\m+\a*\m-1)/(\m+1)}}
\pgfmathsetmacro\g{{\a*\f/2}};
\pgfmathsetmacro\h{{(\m*\g-\f)/(\m+1)}}
\pgfmathsetmacro\fp{\l*\f}
\pgfmathsetmacro\hp{\l*\h}
\pgfmathsetmacro\gp{\l*\g}

\draw[color=black, ->] (-\h-\hp-1,0)--(\fp+\hp-\h+1,0);
\draw[color=black, ->] (0,0)--(0,2*\m+2);

\foreach \k in {1,...,\m}
\draw[color=blue, line width=3] 
    ({(\k-1)*\g-\k*\h},\k)--({\k*\g-\k*\h},\k);

\foreach \k in {1,...,\m}{
\draw[color=black, line width=0.5] 
    ({(\k-1)*\g-\k*\h},\k-0.1) rectangle ({(\k-1)*\g-(\k-1)*\h},\k+0.1); 
\draw[color=black, line width=0.5] 
    ({\k*\g-\k*\h-\h},\k-0.1) rectangle ({\k*\g-\k*\h},\k+0.1); 
}

\draw[color=red, line width=3] (0,0)--(\f,0);

\foreach \k in {1,...,\m}
\draw[color=blue, line width=3] 
    ({(\k-1)*\gp-\k*\hp-\h},\k+\m+1)--({\k*\gp-\k*\hp-\h},\k+\m+1);

\foreach \k in {1,...,\m}{
\draw[color=black, line width=0.5] 
    ({(\k-1)*\gp-\k*\hp-\h},\k-0.1+\m+1) rectangle ({(\k-1)*\gp-(\k-1)*\hp-\h},\k+0.1+\m+1); 
\draw[color=black, line width=0.5] 
    ({\k*\gp-\k*\hp-\h-\hp},\k-0.1+\m+1) rectangle ({\k*\gp-\k*\hp-\h},\k+0.1+\m+1); 
}

\draw[color=red, line width=3] (-\h,\m+1)--(\fp-\h,\m+1);

\draw[color=red, dashed, line width=3] ({\g-\h},0.1)--({\g-\h}, 2-0.1)--({\g-2*\h},2-0.1)--({\g-2*\h}, \m+1);

\node [below] at (\f/2,0) {$f$};
\node [below] at (-\h/2,1) {$h$};
\node [below] at (-\hp/2-\h,\m+2) {$h'$};
\node [below] at (\fp/2-\h,\m+1) {$f'$};
\node [below] at (\f+\h-\g/2,\m) {$g$};

\draw[color=black, line width=0.5, <-]
(\f+\h-\g,\m-1+0.7)--(\f+\h-\g/2-0.2,\m-1+0.7);
\draw[color=black, line width=0.5, ->]
(\f+\h-\g/2+0.2,\m-1+0.7)--(\f+\h,\m-1+0.7);

\node [below] at (\fp+\hp-\gp/2-\h,2*\m+1) {$g'$};

\draw[color=black, line width=0.5, <-]
(\fp+\hp-\gp-\h,2*\m+0.7)--(\fp+\hp-\h-\gp/2-0.2,2*\m+0.7);
\draw[color=black, line width=0.5, ->]
(\fp+\hp-\gp/2+0.2-\h,2*\m+0.7)--(\fp+\hp-\h,2*\m+0.7);

\draw [decorate,decoration={brace}, thick] (\fp+\hp+1,\m+0.2)--(\fp+\hp+1,1-0.2);
\node at (\fp+\hp+1.5,{(\m+1)/2}) {$\blue$};

\draw [decorate,decoration={brace}, thick] (\fp+\hp+1,2*\m+1+0.2)--(\fp+\hp+1,\m+2-0.2);
\node at (\fp+\hp+1.5,{\m+1+(\m+1)/2}) {$\blue'$};

  
\end{tikzpicture}
\end{center}
\caption{The comparison process with $\alpha=1$ and $m=3$, and so $\lambda=5/4$. The overhangs of length $h$ and $h'$ are framed by black rectangles. The dashed line illustrates one shortest path on which red can break through the bottom
configuration $\blue$ of blue sites. 
Once that happens, 
we assume that the entire interval of length $f'$ is painted red, and 
then the red can proceed to break through the next, larger, configuration 
of blue sites, which is the appropriately translated configuration $\blue'$.
} 
\label{fig:comparison}
\end{figure}
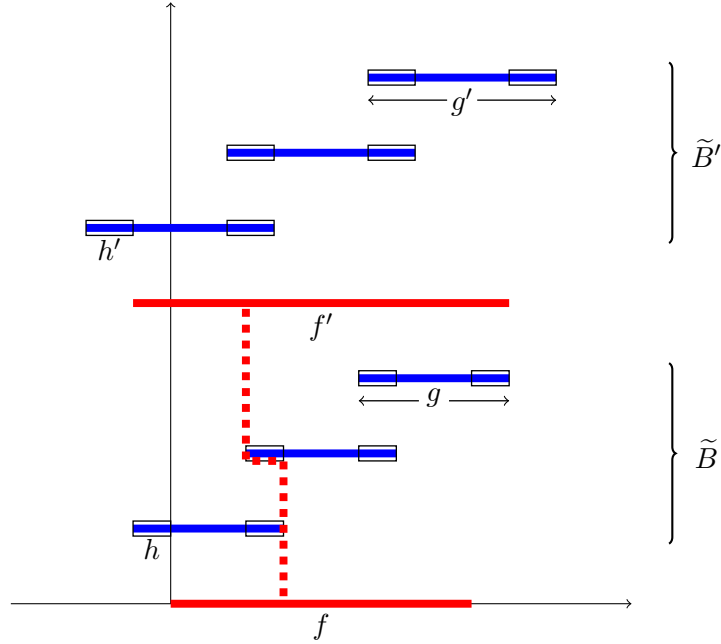

In the discrete dynamics, we will have a sequence of configurations of blue sites, and each configuration is intended to slow the vertical progress of red by some amount. To obtain a lower bound on how much blue can slow the progress of red, when red first reaches a blue configuration, we stop the growth of blue sites in that configuration, but allow blue to continue to spread in configurations above, which red has not reached up to that time. Therefore, in our continuous comparison process, we put a second configuration of blue above the first to further impede the vertical progress of red, which has had the advantage of being able to ``grow'' for an additional time $\alpha h$ before red is first able to reach these blue sites, but now must block an interval of red with horizontal extent of length $f+2h$. To quantify this situation, with $f',g',h'>0$ and $m$ as above, we shift back to the origin (to the right by $h$ and vertically so that red lies on the $x$-axis) to obtain the new initial coloring,
\begin{equation}{\label{continuous_red_prime}}
\red'_0 = [0,f'] \times \{0\}
\end{equation}
and
\begin{equation}\label{continuous_blue_prime}
\blue' = \bigcup_{k=1}^m \Big\{[(k-1)g'-kh',kg'-kh'] \times \{k\}\Big\}.
\end{equation}
We have the additional consistency equations,
\begin{equation}\label{eq:consistency}
m(g'-h') = f'+h', \qquad f' = f+2h, \qquad 
g' = g+\alpha h.
\end{equation}
In matrix form, these are
$$
\begin{pmatrix}
f'\\
g'
\end{pmatrix}
 = A\begin{pmatrix}
f\\
g
\end{pmatrix}
$$
where
$$
A = \begin{pmatrix}
1-\frac{2}{m+1} & \frac{2m}{m+1}\\
\frac{-\alpha}{m+1} & \frac{m+\alpha m+1}{m+1}
\end{pmatrix}.
$$
The matrix $A$ has eigenvalues $1$ and $\lambda=\frac{m+\alpha m -1}{m+1}$, and we want $\lambda>1$, so we choose $m>2/\alpha$. The eigenvector corresponding to  eigenvalue $\frac{m+\alpha m -1}{m+1}$ is $(1, \frac{\alpha}{2})$, so if we assume $g = \frac{\alpha}{2} f$, then
$$
h = \frac{(\frac{m\alpha}{2}-1)f }{m+1} >0
$$
and
$$
(f',g',h') = \lambda(f,g,h).
$$
See Figure~\ref{fig:comparison} for an illustrative example. 

In our discrete system, to make the existence of a blocking configuration of blue sites likely we will need larger intervals at each level wherein existence of initially blue sites guarantee a configuration resembling $\blue'$ will exist at an appropriate time in the growth process. To identify these intervals, we imagine what happens if red grows at a slower speed than $1/\alpha$, and take $\balpha>\alpha$. 
 Then $\red_t$ will first pass the configuration $\blue$ only after time $\balpha h>\alpha h$, so the blue in the next configuration has additional time $(\balpha-\alpha)h$ to grow, relative to the faster-moving red. Therefore, $\bar g' = g'+(\balpha - \alpha) h$ is the length of each blue interval in the second configuration of blue for the slower-moving red. We can translate to the right the right endpoint of each blue interval from $\blue'$ by up to $(\balpha - \alpha) h$ and still blue will have sufficient time to cover the original intervals.
 
\subsection{Discrete blocking configuration}
To make the connection with the discrete system, we construct layers in which blue sites will dominate the blue intervals in the continuous configurations at the appropriate times. In each discrete layer, we will increase the horizontal dimension by a factor $\mu^2>1$ and decrease the vertical dimension by the factor $1/\mu<1$. Since the continuous process indicates that each layer grows horizontally by a factor $\lambda$ from the previous layer, we take $\mu^2 = \lambda$, so the area of each region increases by the factor $\mu = \sqrt{\lambda}>1$. The first layer has height $m/\sqrt{p}$, which is subdivided into $m$ sublayers of height $1/\sqrt{p}$. Each subsequent layer will have height $(1/\mu)$ times the height of the previous layer, rounded up to the next integer multiple of $m$. Within the $k$th layer, each of the $m$ sublayers will contain a blue point within a rectangle with dimensions $(1/\sqrt{\lambda^k p})\times (\lambda^k C/\sqrt{p})$ centered at certain points. The construction will be such that the probability that every such rectangle (for every $k\ge 0$) contains at least one initially blue point  tends to $1$ as $C\to\infty$ for all small $p$.




Construction of the \textbf{blocking configuration centered at $(0,0)$} begins by choosing $\alpha<\balpha<r$ with $\alpha< 1$, fixing $m> 2/\alpha$ and defining a sequence $(f_\ell, g_\ell, h_\ell)_{\ell\ge0}$ for the continuum dynamics with $f_0 = 1$,  $g_0 = \frac{\alpha}{2}$ and $h_0 = \frac{(\frac{m\alpha}{2}-1) }{m+1}$  so that
$$
(f_{\ell+1}, g_{\ell+1}, h_{\ell+1}) = \lambda (f_\ell, g_\ell, h_\ell).
$$
We also define $S_0 = 0$ and $S_{\ell+1} = S_\ell - h_\ell$ for $\ell\ge 1$ to account for the artificial shift to the origin in our frame of reference in the continuous dynamics at~\eqref{continuous_blue_prime}.

For $u\in\bR$, we define the \textbf{rescaled} $u$ to be  $\lceil\frac{1}{\sqrt{p}} u\rceil$, and analogously for vectors $u\in \bR^d$ and for the dimensions of rectangles.
We assume that red starts from a single point at rescaled $(0, -\frac{g_1}{\alpha})$. The shift downward by $g_1/\alpha$ is to allow time for blue to grow horizontally by at least $g_1$ at the time red reaches the $x$-axis and is contained in the rescaled interval $[0,f_0]\times\{0\}$ to initialize the continuous growth comparison. Note that for this initialization step, we assume that red propagates vertically and horizontally at speed $1/\alpha$. After this initialization step, we assume that red propagates vertically at infinite speed, as it does in the continuous comparison process.

We now translate the continuous blocking construction into a discrete one by identifying boxes through which horizontal lines of blue sites will appear by the appropriate times, assuming there are initially blue sites nearby the right edges of these lines. These boxes are divided into \textbf{layers} of $m$ boxes each. The $0$th layer of boxes have lower-left endpoints at rescaled $x$-coordinates $(k-1)g_0-kh_0$ for $k=1, \ldots, m$; they have width rescaled $g_0$; they have height rescaled $1$; and they are stacked disjointly and adjacently above the $x$-axis. Each subsequent layer is stacked disjointly above the previous layer and adjacent to it. The lower-left endpoints of the boxes in the $\ell$th layer are at rescaled $(k-1)g_\ell-kh_\ell + S_\ell$ for $k=1,\ldots, m$; they have width rescaled $g_\ell$; they have height rescaled $1/\lambda^{\ell/2}$; and they are stacked disjointly and adjacently.

\begin{figure}
\begin{center}
\begin{tikzpicture}
 
\def\m{3}
\def\g{5}
\def\h{2}
\def\del{0.5}
\def\a{{0.3, 0.8, 0.4}}
\def\b{{0.7, 0.2, 0.6}}

\foreach \i in {1,...,\m}{
 \fill[lgrey] ({\i*\g-\i*\h},{(\i-1)*\del}) rectangle 
    ({\i*\g-\i*\h+1.2},\i*\del);
    \draw[color=blue, line width=1] 
    ({(\i-1)*\g-\i*\h},{(\i-1)*\del}) rectangle ({\i*\g-\i*\h},{\i*\del});   
     
    \fill [blue] ({(\i*\g-\i*\h)+\a[\i-1]*(-(\i*\g-\i*\h)+\i*\g-\i*\h+1.2)}, 
    {(\i-1)*\del+\b[\i-1]*(-(\i-1)*\del+\i*\del}) circle (2pt);
     
  }
\end{tikzpicture}
\end{center}
\caption{An example of a layer with $m=3$ in which all  boxes (outlined in blue) are successful: there is a blue site in each of the 
$m$ shaded activation regions.} 
\label{fig:success}
\end{figure}
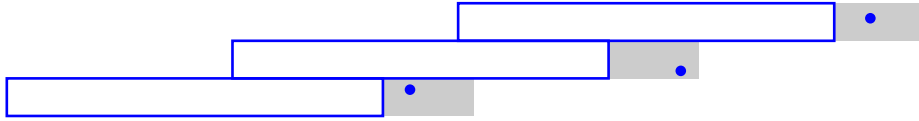

We call the $k$th box in the $\ell$th layer \textbf{successful} if there is at least one blue site initially in the box immediately to its right with the same rescaled height $1/\lambda^{\ell/2}$ and rescaled width $(\balpha-\alpha)h_{\ell-1}$. We call these boxes the \textbf{activation regions} corresponding to the $k$th box in the $\ell$th layer. We call the $k$th box in the $0$th layer successful if there is a blue site initially in the activation region immediately to its right with rescaled height $1$ and rescaled width $\frac12(\balpha-\alpha)$. We say the entire blocking configuration centered at $(0,0)$ is \textbf{successful} if every box in every layer is successful. See Figure~\ref{fig:success}.

Let
\begin{equation*}
\begin{aligned}
    \cC(x_0,y_0) &:= (x_0, y_0) + \{(x,y) : y\le - |x|\}  \\
    \cC'(x_0,y_0) &:=  (x_0, y_0) + \{(x,y) : y\ge  |x|\}
\end{aligned}
\end{equation*}
denote the downward and upward affine cones with peak at $(x_0,y_0)$.
\begin{lemma}\label{successful_blocking_lemma}
    If the blocking configuration centered at $(0,0)$ is successful and initially red sites are confined to the rescaled affine cone $\cC(0, -g_1/\alpha)$, then every box in the construction contains a horizontal line spanning the box that is colored blue before the first time that red enters the smallest horizontal strip containing the box.  Moreover, the static initial configuration of blue, which consists only of the line segments crossing each box in the blocking configuration and does not grow, also prevents red from reaching the smallest horizontal strip containing the boxes in layer $\ell$ before time $\frac{1}{\sqrt{p}}\left[g_{\ell} + (\balpha-\alpha) h_{\ell-1} \right]$.
\end{lemma}

\begin{proof}
    By monotonicity, we may assume that every point in the affine cone is initially red. For the initialization step, observe that it takes time $\frac{1}{\sqrt{p}}\left[g_0 + \frac12(\balpha - \alpha)\right] = \frac{\balpha}{2\sqrt{p}}$ for blue to grow across the boxes in layer $0$. Meanwhile, it takes time 
    $$
    \frac{g_1 r}{\alpha \sqrt{p}}>\frac{r}{2\sqrt{p}}>\frac{\balpha}{2\sqrt{p}}
    $$
    for red to first reach the $x$-axis, which it first does at the origin. 
    This verifies the base case for our induction argument. 
    At the time at which red first reaches the $x$-axis, we can assume the dynamics proceeds from the single red point at the origin and the blue configuration at this time, which we can compare to the continuous dynamics.
    
    For the induction step, we assume that the time it takes for red to first reach layer $\ell$ from this initial configuration is at least $\frac{g_\ell}{\sqrt{p}}$, and the points in layer $\ell$ that are first colored red are contained in the rescaled interval $[S_\ell, S_\ell+f_\ell]$. To compute the time for red to reach the $(\ell+1)$st layer in the discrete dynamics, we compare with the continuous dynamics where red grows at speed $1/\balpha$ within layer $\ell$. In these dynamics, the time for red to pass the $\ell$th layer is $\balpha h_\ell$, so the time for red to reach the $(\ell+1)$st layer in the discrete dynamics is at least
    $$
    \frac{1}{\sqrt{p}}[g_\ell + \balpha h_{\ell}] = \frac{1}{\sqrt{p}}\left[g_{\ell+1} + (\balpha-\alpha) h_\ell \right], 
    $$
    which is an upper bound on the time for blue to horizontally cross the boxes in layer $\ell+1$. This allows comparison with the continuous dynamics to proceed in layer $\ell+1$, and completes the induction.
\end{proof}

The \textbf{blocking configuration centered at $(x,y)\in \bZ^2$} is defined by shifting the blocking configuration centered at $(0,0)$ by $(x,y)$, and as above, a blocking configuration is successful if every box in every layer is successful. We also define the \textbf{transposed blocking configuration centered at $(x,y)\in \bZ^2$} by reflecting the blocking configuration centered at $(0,0)$ over the $x$-axis, then shifting by $(x,y)$. Again, a transposed blocking configuration is successful if every box in every layer is successful.

The next lemma confirms that a collection of successful blocking configurations will not be interfered with by the growth of red sites, provided red sites are restricted to a certain region. Because blue sites expand in time, the red dynamics are not additive, but an additive dynamics can be used for comparison.
\begin{lemma}\label{additive_lemma}
    Let $x_1, x_2, \ldots, x_K\in \bR$. Suppose that for every $k=1, \ldots , K$, the blocking configuration centered at rescaled $(x_k,0)$ is successful, and that the initially red points are confined to the union of rescaled affine cones
    $$
    \bigcup_{k=1}^K \cC(x_k, - g_1/\alpha).
    $$
    Then every box in every blocking configuration contains a horizontal line spanning the box that is colored blue before the first time that red enters the smallest horizontal strip containing the box. 
\end{lemma}
\begin{proof}
Lemma~\ref{successful_blocking_lemma} says that if red starts from rescaled $\cC(x_k, -g_1/\alpha)$ and the blocking configuration centered at rescaled $(x_k,0)$ is successful, then the dynamic blue environment achieves a configuration in which each box of the blocking configuration is crossed by a blue line before it comes in contact with red. This implies that we may assume that these blue line segments are present at time $0$ and blue is static thereafter.

To see that this is still true when we introduce multiple red cones and blue blocking configurations, we proceed inductively on the layer $\ell$. Clearly, the first time that red reaches the $x$-axis is still strictly larger than $\frac{\balpha}{2\sqrt{p}}$, which is an upper bound on the time that it takes for blue to cross the boxes in layer $0$. Therefore, we may assume that these boxes are crossed by blue initially, and blue does not grow within the smallest horizontal strip containing layer $0$. Now, we assume the statement holds through layer $\ell$, so all boxes in these layers are assumed to be crossed by blue initially, and that blue does not grow within the half-plane containing layers $0$ through $\ell$ and the initially red points. Then, until the first time red exits this half-plane, the red dynamics are additive, since they proceed according to a first-passage percolation process on a subset of $\bZ^2$. Therefore, by Lemma~\ref{successful_blocking_lemma}, the time it takes for red to first reach layer $\ell+1$ is at least $\frac{1}{\sqrt{p}}\left[g_{\ell+1} + (\balpha-\alpha) h_\ell \right]$, which is an upper bound on the time for blue to cross layer $\ell+1$. This completes the induction.
\end{proof}

\begin{lemma}\label{protection from below lemma}
    Let $\sigma = \frac{2 m\sqrt{\lambda}}{\sqrt{\lambda}-1}$ and let $x_1, x_2, \ldots, x_K\in \bR$. Suppose that for every $k=1, \ldots , K$, the blocking configuration centered at rescaled $(x_k,-\sigma)$ is successful, and that the initially red points are confined to the union of rescaled affine cones
    $$
    \bigcup_{k=1}^K \cC(x_k, -\sigma - g_1/\alpha).
    $$
    Then there are no red points on the $x$-axis through time $C/p$.
\end{lemma}
\begin{proof}
    By Lemmas~\ref{successful_blocking_lemma} and~\ref{additive_lemma}, blue sites will cross the boxes in layer $\ell$ of the successful blocking configurations centered at rescaled $(x_1, -\sigma), \ldots, (x_K,-\sigma)$ before red first enters the smallest horizontal strip containing layer $\ell$. Since blue spreads at speed $1$, it therefore takes time at least $\frac{g_\ell}{\sqrt{p}} = \frac{g_0 \lambda^\ell}{\sqrt{p}}$ for red to reach layer $\ell$. Therefore, if $\ell\ge \frac{\log(1/p)}{\log \lambda}$, then for small $p$ the time for red to reach layer $\ell$ is at least $C/p$; choose $\ell$ to be the smallest such integer. The total height of the blocking configuration centered at $(0,0)$ through layer $\ell$ is at most
    $$
    m\sum_{k=0}^\ell \left\lceil \frac{1}{\lambda^{k/2}\sqrt{p}} \right\rceil \le m(\ell+1) + m\sum_{k=0}^\infty \frac{1}{\lambda^{k/2}\sqrt{p}} = m(\ell+1) + \frac{m\sqrt{\lambda}}{(\sqrt{\lambda}-1)\sqrt{p}} \le \frac{\sigma}{\sqrt{p}}
    $$
    for small $p$, by our choice of $\sigma$. Thus, shifting the blocking configuration down by rescaled $\sigma$ puts the $\ell$th layer below the $x$-axis, and prevents red from reaching the $x$-axis by time $C/\sqrt{p}$. 
\end{proof}

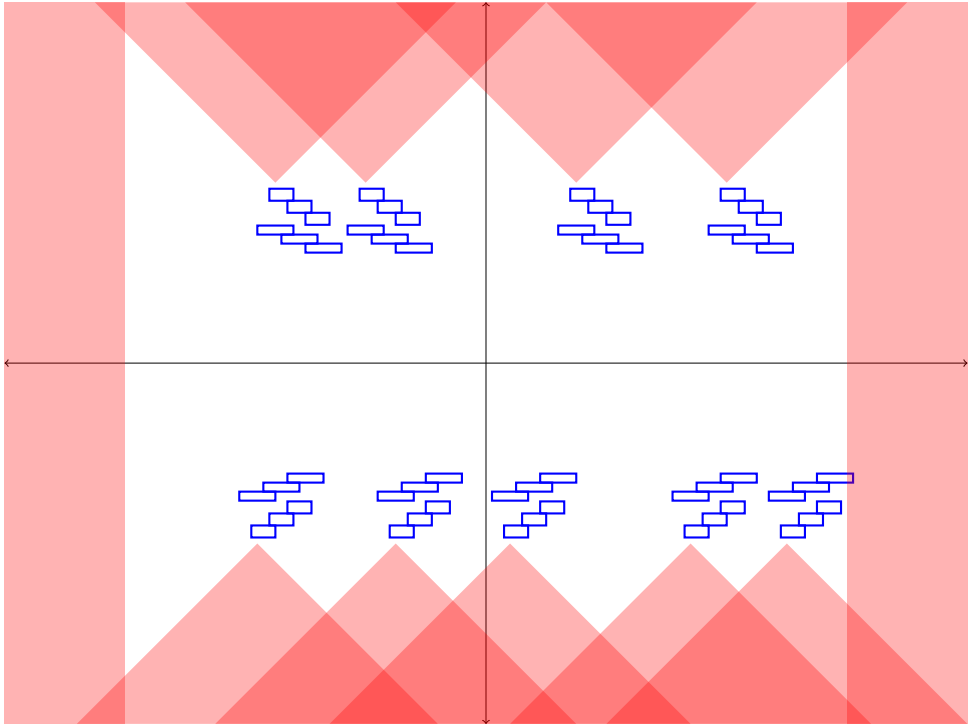
\begin{figure}[ht!]
\begin{center}
\scalebox{0.8}{
\begin{tikzpicture}



\def\a{1}
\def\f{4}
\def\m{3}

\pgfmathsetmacro\l{{(\m+\a*\m-1)/(\m+1)}}
\pgfmathsetmacro\s{{2*\m*sqrt(\l)/(sqrt(\l)-1}}
\pgfmathsetmacro\g{{\a*\f/2}};
\pgfmathsetmacro\h{{(\m*\g-\f)/(\m+1)}}
\pgfmathsetmacro\fp{\l*\f}
\pgfmathsetmacro\hp{\l*\h}
\pgfmathsetmacro\gp{\l*\g}

\draw[color=black, <->] (-8,0)--(8,0);
\draw[color=black, <->] (0, -6)--(0,6);

\foreach \k in {-3.8,-1.5,0.4,3.4,5}
{
\fill[opacity=0.3,red]{(\k-3,-6)--(\k, -3)--(\k+3,-6)};
 
\foreach \i in {1,...,\m}
{
 \draw[color=blue, line width=1] ({(\i-1)*0.4-\i*0.1+\k},\i*0.2-3-0.1) rectangle ({\i*0.4-\i*0.1+\k},\i*0.2-3+0.1);

\draw[color=blue, line width=1] 
({(\i-1)*0.6-\i*0.2+\k-0.1},\m*0.2+\i*0.15-3-0.075+0.075/2) rectangle ({\i*0.6-\i*0.2+\k-0.1},\m*0.2+\i*0.15-3+0.075+0.075/2);
} 
}

\foreach \k in {-3.5,-2,1.5,4}
{
\fill[opacity=0.3,red]{(\k-3,6)--(\k, 3)--(\k+3,6)};
 
\foreach \i in {1,...,\m}{
 
    \draw[color=blue, line width=1] 
    ({(\i-1)*0.4-\i*0.1+\k},-\i*0.2+3+0.1) rectangle ({\i*0.4-\i*0.1+\k},-\i*0.2+3-0.1);
  
    \draw[color=blue, line width=1] ({(\i-1)*0.6-\i*0.2+\k-0.1},-\m*0.2-\i*0.15+3+0.075-0.075/2) rectangle ({\i*0.6-\i*0.2+\k-0.1},-\m*0.2-\i*0.15+3-0.075-0.075/2);
}    
}

\fill[opacity=0.3,red] (-8,-6) rectangle (-6,6);
\fill[opacity=0.3,red] (6,-6) rectangle (8,6);
  
\end{tikzpicture}
}
\end{center}
\caption{Protection of the origin from red occupation, as 
described in Lemma~\ref{protection lemma}.  
Red is initially confined to the area shaded red. The blue
rectangles are the first two layers of each successful 
blocking configuration: each contains a blue site in its activation region  and is horizontally traversed  by a blue line by the 
time it contains a red site. The closest vertical  lines containing red points are
at distance on the order $1/p$ from the $y$-axis, while the vertices 
of the cones are at distance on the order $1/\sqrt p$ from the $x$-axis.
} 
\label{fig:protection}
\end{figure}

We now state our final protection lemma, which is illustrated in 
Figure~\ref{fig:protection}.

\begin{lemma}\label{protection lemma}
    Let $\sigma = \frac{2 m\sqrt{\lambda}}{\sqrt{\lambda}-1}$ and let $x_1, x_2, \ldots, x_K\in \bR$ and $x_1', \ldots, x_{K'}'\in\bR$. Suppose that for every $k=1, \ldots , K$, the blocking configuration centered at rescaled $(x_k,-\sigma)$ is successful, and for every $k=1, \ldots, K'$, the transposed blocking configuration centered at rescaled $(x_k',\sigma)$ is successful, and all initially red points are either in the union of rescaled affine cones, 
    $$
    \bigcup_{k=1}^K \cC(x_k, -\sigma - g_1/\alpha) \cup \bigcup_{k=1}^{K'} \cC'(x_k',\sigma+g_1/\alpha),
    $$
    or in the complement of a vertical strip around the origin,
    $$
    \left([-3C/(rp),3C/(rp)]\times \bZ \right)^c.
    $$
    Then there are no red points in $[-C/p,C/p]\times \{0\}$ at time $C/p$.
\end{lemma}
\begin{proof}
    If we only allow red points initially in the union of rescaled affine cones, then Lemma~\ref{protection from below lemma} and the observation that the red points above and below the $x$-axis do not interact by time $C/p$ implies that there are no red points anywhere on the $x$-axis by time $C/p$. Now observe that by a standard speed-of-light argument, the state of a vertex $(x,0)$ for $x\in [-C/p,C/p]$ at time $C/p$ depends only on the initial colors of sites within distance $C(\frac1r +1)/p\le 2C/(rp)$. Therefore, allowing red (or blue) points outside of the strip $[-3C/(rp),3C/(rp)]\times \bZ$ will not change the states of the vertices in $[-C/p,C/p]\times \{0\}$ at time $C/p$.
\end{proof}

Lemma~\ref{protection lemma} describes a region that needs
to be void of initially red sites to protect the origin
from red occupation up to time on the order $1/p$. This region will be random --- dependent on 
the initial configuration of blue sites --- and we 
proceed to show that it is 
unlikely to be too large.

For every $k\in\bZ$, we define $I_k$ to be the indicator of the event that there is a successful blocking configuration at rescaled $(k,-\sigma)$, with $\sigma$ as defined in Lemma~\ref{protection from below lemma}. Then, for each $k\in\bZ$, we define $Z_k = \inf \{\ell\ge 0 : I_{k+\ell}=1\}$ to be the (rescaled) distance from $k$ to the next successful blocking site to its right.

\begin{lemma}\label{lem:finite gaps}
    Let $B>0$. Assume the initial probability of a vertex being blue is $Bp$. Then there exists a large enough $B$ such that 
    $$
    \sup_{0<p\le 1/B} \bE Z_0<\infty.
    $$
\end{lemma}
\begin{proof}
    For $z\in \bZ$, let $P_z = \prob{Z_0\ge z}$ be the probability that there are $z$ consecutive $0$s starting at $0$: $I_0= I_1=\cdots = I_{z-1}=0$. For $\ell\ge 0$, let $A_\ell$ be the event that the $\ell$th layer of the blocking configuration centered at $(0,-\sigma)$ is not successful, and each layer $\ell'<\ell$ is successful. Partitioning according to these events, we have for $z\ge 1$,
    \begin{equation}\label{eq:layer partition}
        P_z = \sum_{\ell=0}^\infty \prob{A_\ell} \prob{Z_0\ge z | A_\ell} \le \sum_{\ell=0}^\infty \prob{A_\ell} \prob{Z_0\ge z - \lceil D\lambda^{\ell}\rceil}
    \end{equation}
    where $D$ is a constant such that $D\lambda^{\ell}>(\balpha-\alpha)h_{\ell-1}$ for every $\ell\ge 1$. The last inequality follows because on the event $A_\ell$, the layers in the blocking constructions strictly above layer $\ell$ remain unaffected by the conditioning, and the layer $\ell'\le \ell$ has activation regions with rescaled width at most $\lceil (\balpha - \alpha)h_{\ell-1}\rceil$. Therefore, we have that $I_k, I_{k+1}, \ldots$ are independent of $A_\ell$ whenever $k\ge \lceil (\balpha - \alpha)h_{\ell-1}\rceil$.
    
    The event $A_\ell$ implies that there is a box in the $\ell$th layer of the blocking configuration centered at $(0,-\sigma)$ that is not successful. By a union bound on the $m$ activation regions in the $\ell$th layer, for $\ell\ge1$ this probability is at most
    \begin{equation*}
    \prob{A_\ell}\le m (1-Bp)^{\lambda^{-\ell/2}\cdot (\balpha-\alpha)h_{\ell-1} / p} \le m\exp\left[-B \delta \lambda^{\ell/2}\right],
    \end{equation*}
    for a small enough constant $\delta>0$. By possibly decreasing $\delta$, we have
    \begin{equation}\label{eq:layer prob}
    \prob{A_\ell} \le m\exp\left[-B \delta \lambda^{\ell/2}\right]
    \end{equation}
    for all $\ell\ge 0$. Combining inequalities~\eqref{eq:layer partition} and \eqref{eq:layer prob} gives
    \begin{equation}\label{eq:prob recursion}
        P_z \le m\sum_{\ell=0}^\infty \exp\left[-B \delta \lambda^{\ell/2}\right] P_{z-\lceil D\lambda^{\ell}\rceil}.
    \end{equation}

    We now claim that we can choose $B$ sufficiently large so that $P_z \le 2^{-\sqrt{z}}$ for all $z\ge 0$. Clearly, for $z\le 0$ we have $P_z = 1 = 2^0$, and we proceed by induction. By~\eqref{eq:prob recursion} and the induction hypothesis, we have
    \begin{equation}\label{eq:prob recursion solved}
        \begin{aligned}
          P_z &\le m\sum_{\ell=0}^\infty \exp\left[-B \delta \lambda^{\ell/2}\right] 2^{-\sqrt{\max(z-2D\lambda^\ell,0)}} \\
          &\le m\sum_{\ell=0}^\infty \exp\left[-B \delta \lambda^{\ell/2}\right] 2^{-\sqrt{z}+\sqrt{2D\lambda^\ell}} \\
          &= 2^{-\sqrt{z}}\cdot m\sum_{\ell=0}^\infty \exp\left[-\lambda^{\ell/2}\left(B \delta - \sqrt{2D}\log2\right)\right].
        \end{aligned}
    \end{equation}
    The series in the last line converges for large enough 
$B$, and decreases to $0$ as $B\to\infty$, so we choose $B$ such that the series is smaller than $1/m$, which justifies the induction step, and proves the claim. It follows that $\bE Z_0 = \sum_{z\ge1} P_z \le \sum_{z\ge 1} 2^{-\sqrt{z}}$ for all $p\in (0,1/B]$.
\end{proof}

The next lemma provides the key step in the proof of Theorem~\ref{thm:1vrho}. Due to the rescaling, 
the size of the red-free area will have the area of the set $F$, defined below, multiplied by $1/p$.

\begin{lemma}\label{lem:red-free area}
    Assume the probability of a vertex being initially blue is $Bp$ and $B$ is chosen as in Lemma~\ref{lem:finite gaps}. Let $\cI = \{k\in \bZ : I_k=1, -3C/(r\sqrt{p}) \le k\le 3C/(r\sqrt{p})\}$ and let
    $$
    F = \left(\bigcup_{k\in \cI} \cC(k,-\sigma-g_1/\alpha)\right)^c \cap ([-3C/(r\sqrt{p}),3C/(r\sqrt{p})]\times (-\infty,0]).
    $$
    For any $\epsilon>0$, we can choose either $C>0$ sufficiently small (depending on $w$) or $w$ sufficiently large (depending on $C$) so that
    $$
    \prob{\mathrm{Area}(F) \ge w/\sqrt{p}} <\epsilon
    $$
    for all sufficiently small $p>0$.
    \end{lemma}
\begin{proof}
    Observe that for $k\ge 1$, the region below the $x$-axis, above the cones $\cC(0,-\sigma-g_1/\alpha)$ and $\cC(k,-\sigma-g_1/\alpha)$, and between the lines $x=0$ and $x=k$ has area
    $$
    k(\sigma + g_1/\alpha) + k^2/2 \le Dk^2
    $$
    for a constant $D>0$. Moreover, if $I_0=I_1=\ldots=I_{k-1}=0$, then $\sum_{j=0}^{k-1}Z_j \ge \frac12 k^2$. By decomposing the sequence $(I_k)$ into gaps between successive $1$s, we conclude that on the event $\mathrm{Area}(F)<\infty$, we have
    $$
    \mathrm{Area}(F)\le 2D \sum_{|k|\le 3C/(r\sqrt{p})} Z_k.
    $$
    Noting that the area of $F$ is finite if, say, $Z_0\le C/\sqrt{p}$, the conclusion now follows from Lemma~\ref{lem:finite gaps} and Markov's inequality.
\end{proof}

\subsection{Proof of Theorem~\ref{thm:1vrho} for $\rho=1$}

\begin{proof}[Proof of parts (2) and (3) of Theorem~\ref{thm:1vrho} when $\rho=1$]
Lemma~\ref{lem:red wins} implies part (2). To prove part (3), we may replace $p$ by $Bp$, where the constant $B$ is chosen in Lemma~\ref{lem:finite gaps}. Fix an $\epsilon>0$, and choose $w = w(\epsilon,C)$ such that the area of the region $F$ defined in Lemma~\ref{lem:red-free area} satisfies
$$
\P(\mathrm{Area}(F)\ge w/\sqrt{p})<\epsilon.
$$
Let $H = \{\mathrm{Area}(F)\ge w/\sqrt{p}\}$ be the event above, and let $H'$ be the same event after the entire initial configuration has been reflected across the $x$-axis, so that $\P(H') = \P(H) <\epsilon.$ Let $F'$ be the region specified in the reflected event $H'$, which is the region protected by transposed blocking configurations in the analogous strip above the $x$-axis. Then, on the event $(H\cup H')^c$, if the rescaled region $F\cup F'$ is devoid of red points, then the conditions of Lemma~\ref{protection lemma} are met, so there will be no red points in $[-C/p,C/p]\times \{0\}$ at time $C/p$. Therefore,
\begin{equation}
    \begin{aligned}
        \P(0 \text{ is not eventually blue}) &\le \P(H\cup H')  \\
        &\quad + \P(\{\text{there is a red point in rescaled } F\cup F'\}\cap (H\cup H')^c) \\
        &\quad + \P(\text{there is no blue point in } (0, C/p]\times \{0\})\\
        &\le 2\epsilon + q(2w/p^{3/2}) + e^{-BC} \\
        &= 2\epsilon + 2aw B^{3/2} + e^{-BC}\\
        &< 4 \epsilon
    \end{aligned}
\end{equation}
by first choosing $C$ large, which along with $\epsilon$ determines $w$, then choosing $a$ sufficiently small. This completes the proof of (2) in the case where $a\to0$.

Now, fix $a>0$ and replace $p$ by $(B+1)p$. We will make a sprinkling argument, so we first assume that sites are occupied by blue with probability $Bp$ or red with probability $q = a((B+1)p)^{3/2}$. Choose $w< 1/(8a(B+1)^{3/2})$, and let $H$ and $H'$ be the same events as above. Choose $C>0$ sufficiently small so that $\probsub{Bp}{H\cup H'}<1/4$. Now, independently of everything, we add initially blue points with probability $p$ along the $x$-axis. Then, we have
\begin{equation}
    \begin{aligned}
        &\probsub{(B+1)p}{\text{0 is eventually blue}}\ge\\
        &\quad \left(1-\probsub{Bp}{H\cup H'} - \probsub{Bp}{\{\text{there is a red point in rescaled } F\cup F'\}\cap(H\cup H')^c} \right) \times\\
        &\qquad \probsub{p}{\text{there is a blue point in }(0,C/p]\times\{0\}}\\
        &\ge (1/2)(1-e^{-C})>0
    \end{aligned}
\end{equation}
for all small enough $p$. This completes the proof of the statement (2) for arbitrary fixed $a$.
\end{proof}

\section{Larger range for red}\label{sec:long range}

In this section, we provide the necessary modification 
of our arguments to prove Theorems~\ref{thm:1vrho} and~\ref{thm:1d long-range} in full generality.

\begin{proof}[Proof of Theorem~\ref{thm:1vrho}]
To prove the statement (2), we need an analog of Lemma~\ref{lem:red wins}. Now, in order to stop red from growing, blue must form a horizontal strip of $\rho$ neighboring lines, which red cannot jump over. The vertical scale is now $p^{-\rho/(\rho+1)}$, and the new version of the event $H$ is 
\begin{equation*}
\begin{aligned}
    H &= \{\text{there is an initially red site } (x,y)\in [0,\epsilon/2p]\times [0,\epsilon / p^{\rho/(\rho+1)}] \text{ and} \\
    & \qquad \text{there is no set of } 
    \rho \text{ neighboring horizontal lines in }\\
    &\qquad [x-2\epsilon \tau r /p^{\rho/(1+\rho)}, x+2\epsilon \tau r/p^{\rho/(1+\rho)}] \times [0,\epsilon/p^{\rho/(1+\rho)}] \\
    &\qquad \text{such that each contains a potentially blue point}\}.
\end{aligned}
\end{equation*}
Let $G$ be defined as in~\eqref{eq:GH}. It follows that the event $G\cap H$ implies that the origin is eventually red, and if $H_1 = \{\text{there is an initially red site in } [0,\epsilon/2p]\times [0,\epsilon / p^{\rho/(\rho+1)}]\}$, then
$$
P(H^c \mid H_1) \le \epsilon p^{-\rho/(\rho+1)}\left(4\epsilon \tau r p^{-\rho/(\rho+1)}p\right)^\rho=(4\tau r)^\rho\epsilon^{\rho+1}.
$$
The rest of the proof is analogous to that of Lemma~\ref{lem:red wins}.

For the statement (3), the proof proceeds along the lines of the $\rho=1$ case. The blocking construction is the same, with the following modified definition of a successful box. A box is now called \textbf{successful} if its activation region has $\rho$ neighboring rows, each of which contains at least one blue site. We need to replace Lemma~\ref{lem:finite gaps} with a bound on $\bE Z_0$ for the gaps between consecutive blocking structures for the range-$\rho$ process, and we do so below. The rest of the proof is the same.

Assume we have a box with width $w$ and height $h$, where 
$wBp<1$.  
The probability that such a box is unsuccessful is  
bounded above by 
$$
(1-c(wBp)^\rho)^h\le \exp(-cB^\rho hw^\rho p^\rho),
$$
for a small enough $c>0$. 
Therefore, the upper bound for probability $\P(A_\ell)$ in the 
proof of Lemma~\ref{lem:finite gaps} is now for all $\ell\ge 0$
\begin{equation*}
\begin{aligned}
    \prob{A_\ell} &\le 
    m\exp\left[-cB^\rho \lambda^{-\ell/2}p^{-\rho/(\rho+1)}(\balpha-\alpha)^\rho h_{\ell-1}^\rho p^{-\rho^2/(\rho+1)} \cdot p^{\rho}\right]
    \\
    &\le m\exp\left[-B^\rho \delta \lambda^{(\rho-1/2)\ell}\right],
    \end{aligned}
    \end{equation*}
for a small enough $\delta> 0$. It follows that (\ref{eq:layer prob})
holds for all $\rho\ge 1$, so the computation (\ref{eq:prob recursion solved}) holds as well. 
(In fact, we can get exponential instead of stretched 
exponential tail bounds for $\rho\ge 2$.)
\end{proof}

\begin{proof}[Proof of Theorem~\ref{thm:1d long-range}]
Since it makes no difference at our level of precision, we assume also that blue and red grow at the same rates, so $r=1$.
    We first prove the lower bound. For blue to reach the origin, it is sufficient for there to be an initially blue point in $[0,\delta/p]\times \{0\}$ and for no $\tau$ consecutive red points to reach this segment of the $x$-axis by time $\delta/(\tau p)$. To stop red from above, we subdivide the region $\bZ \times [0,N]$ into $N\times N$ disjoint, adjacent boxes. We call a box \textbf{successful} if it contains $\rho$ consecutive rows each with at least one initially blue point.

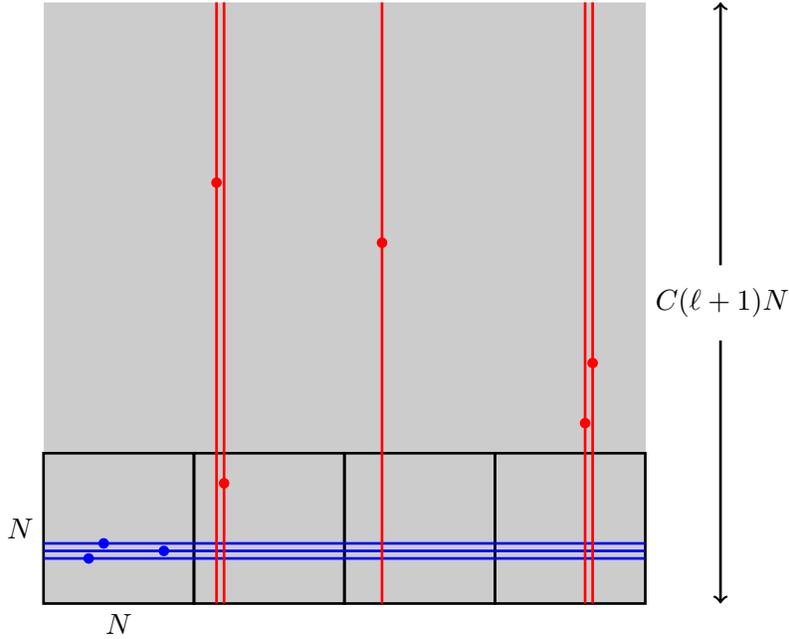
\begin{figure}[ht!]
\begin{center}
\begin{tikzpicture}
  
\fill[lgrey] (0,0) rectangle 
    (8,8); 
  \def\l{4}
  \def\a{{0.3, 0.8, 0.4}}
  \def\by{{0.7, 0.2, 0.6, 0.3, 0.4}}
  \def\bx{{2.3, 2.4, 4.5, 7.2,7.3}}
  \foreach \i in {1,...,\l}{
  \draw[color=black, line width=1] 
  (2*\i-2,0) rectangle (2*\i,2);
  }
  \foreach \i in {1,...,3}{
  \fill [blue] (2*\a[\i-1], {\i*0.1+0.5}) circle (2pt);
   \draw[color=blue, line width=1] (0,{\i*0.1+0.5})--(8,{\i*0.1+0.5});
  }
  \foreach \i in {1,...,5}{
  \fill [red] (\bx[\i-1], 8*\by[\i-1]) circle (2pt);
   \draw[color=red, line width=1] (\bx[\i-1],0)--(\bx[\i-1],8);
  }
  \node [below] at (1,0) {$N$};
  \node [left] at (0,1) {$N$};
  \node[right] at (8,4){$C(\ell+1)N$};
   \draw[color=black, line width=1, ->] (9,4.5)--(9,8);
   \draw[color=black, line width=1, <-] (9,0)--(9,3.5);
  
\end{tikzpicture}
\end{center}
\caption{Protection from red in the proof of the lower bound in Theorem~\ref{thm:1d long-range}, with $\rho=\tau=3$, $Z_k=3$. All initially red sites (circled) in the shaded box are covered by the red lines, no three of which are neighboring. The three initially blue sites (also circled) therefore generate the blue neighboring lines before the 
red from above the box can interfere with them.} 
\label{fig:long range protection}
\end{figure}

    For $k\in\bZ$, we define $I_k$ to be the indicator of the event that the $k$th box is successful. Then, for each $k\in\bZ$, we define $Z_k = \inf \{\ell\ge 0 : I_{k+\ell}=1\}$ to be the number of boxes from the $k$th box to the next successful box to its right. We choose $N$ just large enough so that $EI_k\ge 1/2$, so $N\asymp (1/p)^{\rho/(\rho+1)}$. Then, since boxes are successful independently, it follows that $E(Z_k^{\tau+1}) =E(Z_1^{\tau+1})<\infty$ for all $k$, so by Markov's inequality, for any $\epsilon>0$ we can find $C$ so that
    \begin{equation}\label{eq:gap sizes}
    \P\left(\sum_{k=1}^{\delta/{Np}} (Z_k)^{\tau+1} \ge C\delta/(Np)\right)<\epsilon.
    \end{equation}
    We define the region, $F$, above the $x$-axis, which is unprotected by successful blue boxes, as follows. 
    For each interval $[k,k+\ell]$ such that $I_k = I_{k+\ell}=1$ and $I_{k+1} = \cdots = I_{k+\ell-1}=0$, we include in $F$ the rectangular region with dimensions $(\ell+1)N\times C(\ell+1)N$ that
    shares lower left corner with the $k$th box and lower right corner with the $(k+\ell)$th box. 
    The constant $C$ here is chosen sufficiently large so that the blue lines growing from the successful boxes at $k$ and $k+\ell$ will fill the gap between them before any red starting above the rectangular region can reach height $N$. See~Figure~\ref{fig:long range protection}. For one such region, the probability that it contains $\tau$ neighboring vertical lines, each containing at least one initially red point, is bounded above by $C (\ell N)^{\tau+1} q^\tau$. If 
    $$
    \sum_{k=1}^{\delta/{Np}} (Z_k)^{\tau+1} < C\delta /(Np),
    $$ 
    then the probability that any such rectangle in $F$ has this property is at most $C\delta N^{\tau+1}q^\tau / (Np) \le C\delta a^{\tau}$. For large $a$, this can be made small by choosing $\delta$ small enough. 
    The interval $[0,\delta/p]\times\{0\}$ is protected from below by an analogous construction with positive probability, 
    Therefore, the origin will be eventually blue with positive probability. This shows that the $\liminf$ in the statement is positive for every $a>0$. By taking $\delta$ large, the probability that there is a blue point in $[0,\delta/p]\times\{0\}$ can be made close to $1$, then by taking $a$ sufficiently small, this interval will be protected from above and below with probability close to $1$. This shows that the $\liminf$ tends to $1$ as $a\to 0$.

We now sketch the proof of the upper bound on the $\limsup$ in the statement of the theorem, which  is again analogous to the proof of Lemma~\ref{lem:red wins}. For $x\in \bZ$, define the events
     \begin{equation}\label{eq:GH 1d}
    \begin{aligned}
        G &= \{\text{there is no potentially blue site in } [-\epsilon \tau /p,\epsilon \tau /p]\times\{0\}\}, \quad \text{and} \\
        H(x) &= \{\text{the $\tau$ consecutive vertical lines $[x,x+\tau-1]\times [0,\epsilon/p^{\rho/(1+\rho)}]$}
        \\
        &\qquad 
        \text{each contain an initially red site}
        \\
        &\qquad\text{and there is no 
        set of $\rho$ neighboring horizontal lines}
        \\
        &\qquad \text{in $[x-2\epsilon \tau/p^{\rho/(1+\rho)}, x+2\epsilon \tau/p^{\rho/(1+\rho)}]\times [0,\epsilon/p^{\rho/(1+\rho)}]$}
 \\
    &\qquad \text{such that each contains a potentially blue point}\}.
    \end{aligned}
    \end{equation}
    Now define 
    \begin{equation}
    H^+ = \bigcup_{x \in [0, \epsilon/2p]} H(x), \quad \text{ and } \quad
    H^- = \bigcup_{x \in [-\epsilon/2p,0]} H(x).
    \end{equation}
    If $H(x)$ occurs, the $\tau$ consecutive vertical lines 
    will eventually produce $\tau$ consecutive red sites on 
    the $x$-axis, which cannot be crossed by blue. Therefore, 
    on the event $G\cap H^+\cap H^-$, the origin never becomes blue. 
    For any fixed $\epsilon>0$, the probability that there exist $\tau$ consecutive vertical lines in $[0,\epsilon/2p]\times [0,\epsilon/p^{\rho/(1+\rho)}]$ and $\tau$ consecutive vertical lines in $[-\epsilon/2p,0]\times [0,\epsilon/p^{\rho/(1+\rho)}]$, each with an initially red site, is bounded away from $0$ for any $a>0$, and tends to $1$ as $a\to\infty$. Moreover, conditional on this event, the probability of $G\cap H^+\cap H^-$ is also bounded away from $0$ for small enough $\epsilon$, and tends to $1$ as $\epsilon\to 0$.
\end{proof}

\begin{proof}[Proof of Corollary~\ref{thm:1d three phases}]
    First, we check that $\gamma'\gamma>1$.
    \begin{equation*}
        \begin{aligned}
            \gamma'\gamma &= \left(\frac{1}{\rho}+\frac{\tau}{\tau+1}\right)\left(\frac{1}{\tau}+\frac{\rho}{\rho+1}\right)\\
            &=1 + \frac{1}{\rho\tau} + \frac{1}{(\rho+1)(\tau+1)}>1.
        \end{aligned}
    \end{equation*}
    Then by Theorem~\ref{thm:1d long-range}, when $q\gg p^\gamma$, the origin is eventually blue with probability tending to $0$. Reversing the roles of blue and red and exchanging $\rho$ and $\tau$, Theorem~\ref{thm:1d long-range} implies that when $p\gg q^{\gamma'}$, the origin is eventually red with probability tending to $0$, and the corollary follows.
\end{proof}

\section{Extensions and open problems}\label{sec:open}

We begin with a natural question on a more precise form of our main two theorems.

 \begin{open}
 Are the $\liminf$ and $\limsup$ in Theorem~\ref{thm:1d long-range} and  Theorem~\ref{thm:1vrho} equal, i.e., does the limit exist?
 \end{open}

Processes with three (or more) colors are a natural generalization. 
Here, we will merely touch on them, with a result and open 
problems on two basic cases. The third color will be green, and 
the initial densities of the blue, red, and green colors will be 
$p_b$, $p_r$, and $p_g$. Also, we will assume that $r=1$ and that 
the neighborhood for the green growth is $\cG\subset\bZ^2$. Then, 
an empty vertex that has a colored neighbor in $x+\cB$, $x+\cR$, 
or $x+\cG$ at time $t$ changes to the corresponding color at time $t+1$, 
with ambiguities resolved at random. 

First consider the case when red and blue colors undergo
one-dimensional nearest neighbor growth in orthogonal directions, as in Corollary~\ref{thm:1d three phases} with $\rho=\tau=1$, 
$\cB=\{\pm e_1\}$, $\cR=\{\pm e_2\}$; whereas green undergoes
two-dimensional nearest-neighbor growth, $\cG=\{\pm e_1,\pm e_2\}$. We also assume that  $p_b=p_r=p$ and $p_g=q$ are small. 
See Figure~\ref{fig:3-color-211} (left). 

\begin{figure}[ht!]
    \centering
    \includegraphics[width=0.45\linewidth]{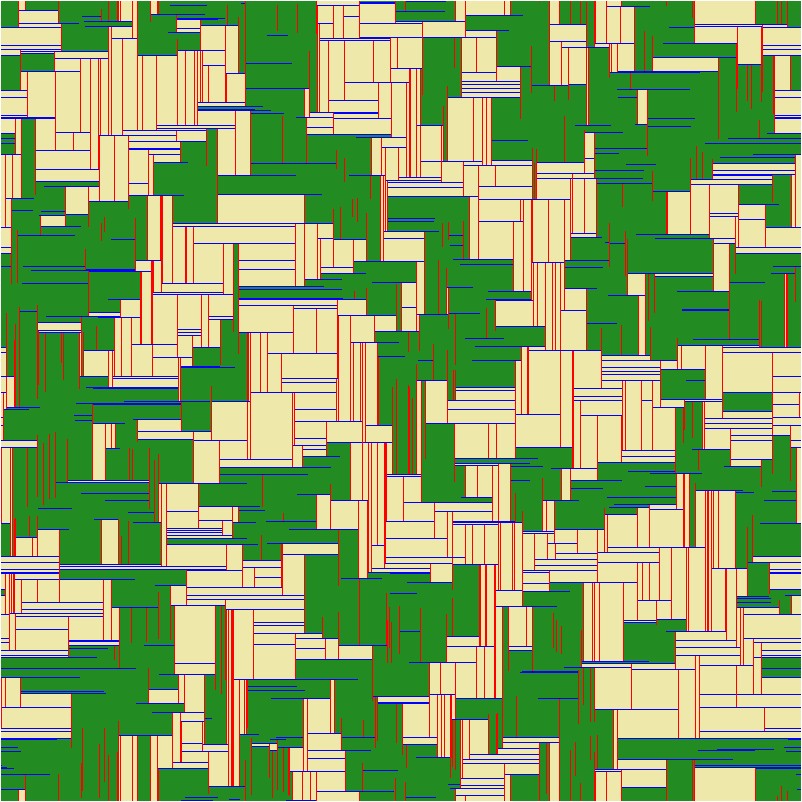}
    \includegraphics[width=0.45\linewidth]{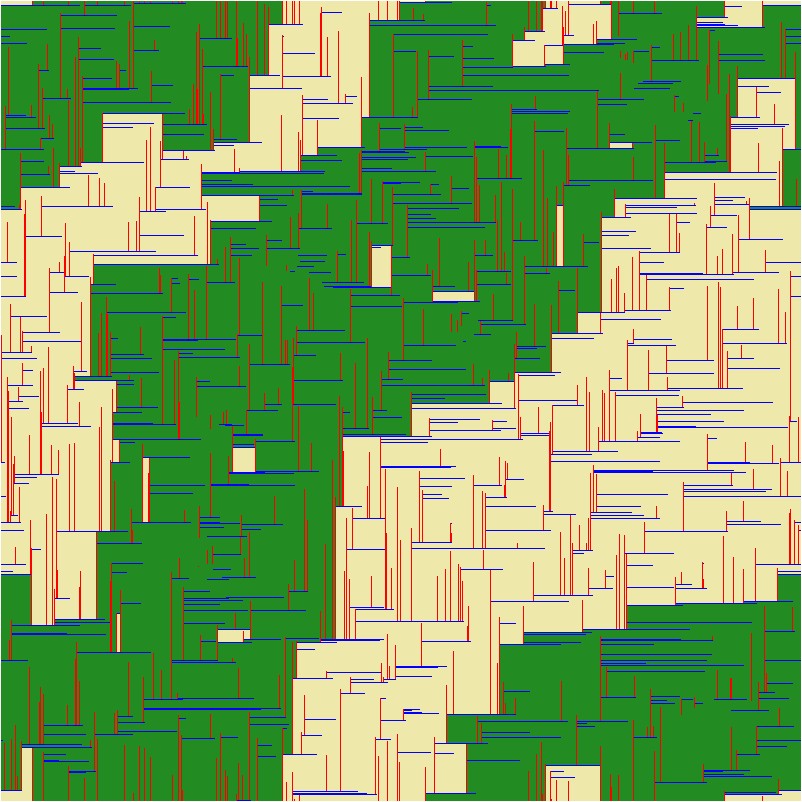}
    \caption{Two cases of three color dynamics, with one-dimensional 
    growth by blue and red and two-dimensional growth by green. Left: $\cB=\{\pm e_1\}$, $\cR=\{\pm e_2\}$, $\cG=\{\pm e_1,\pm e_2\}$, 
    $p_b=p_r=0.001$, $p_g=0.002$. Right: $\cB=\{e_1\}$, $\cR=\{e_2\}$, $\cG=\{\pm e_1,\pm e_2\}$, 
    $p_b=p_r=0.001$, $p_g=0.00001$.}
    \label{fig:3-color-211}
\end{figure}

\begin{theorem}\label{thm:3-color}
If $p\ll q$, then $\P(\text{\rm the origin is eventually green})\to 1$, 
while if $p\gg q$, then $\P(\text{\rm the origin remains empty})\to 1$.
\end{theorem}

\begin{proof} We denote by $B_n$ the  
$n\times n$ box centered at the origin.

The first statement is easy to prove. Indeed, let $n=n(p)$ be such that $n^2p\to0$ while $n^2q\to\infty$ and let $G$ be the event that
$B_{n}$ contains an initially green site, while $B_{2n}$ initially contains neither a red nor a blue site. Then clearly $\P(G)\to 1$ but 
on $G$ the origin becomes green by time $n$. 

For the converse, by Corollary~\ref{thm:1d three phases} with $\tau=\rho=1$, it suffices to show that the origin is not eventually green with high probability. We now assume  $n=n(p)$ is such that $n^2p\to 1$ while $n^2q\to 0$. Partition the lattice into $n\times n$ boxes with 
the origin in the center of one of them. Call a box {\it a red fence\/}
if it contains an initially red site and no initially blue sites and there are no initially blue sites in any of the 
neighboring four boxes (i.e., those that share an edge).
Define {\it a blue fence\/} analogously. We say that a blue line {\it crosses\/} a box if it connects the left and right sides of the box. 

Observe that a red fence will never be crossed by a blue line. Observe also that a box is a red fence with 
probability bounded away from $0$, and that  boxes at distance at least 3 from each other are red 
fences independently.
It follows that the probability that 
the box centered at the origin is the left endpoint of an interval of $m$ boxes, 
none of which is a red fence, is at most $\exp(-\alpha m)$, where $\alpha>0$. 

Assume now that $q=0$. Then, as it cannot have a boundary site with a
horizontal and a vertical neighbor inside, every connected 
component of forever empty sites is a rectangle, bounded on the top and 
bottom by horizontal intervals of blue sites and on the left and right by 
vertical intervals of red sites. Let $R$ be such connected component of the origin with dimensions $A\times B$. 

Assume that the horizontal dimension is maximal: $A\ge B$. 
Then there is a row of at least $\lfloor A/n\rfloor-2$ boxes, 
none of which is a red fence, that intersects the blue boundary of $R$ but not the red boundary.
Then 
$$
P(A\ge B, A\ge mn)\le \sum_{k\ge m-2} 2k^4e^{-\alpha k}
$$
and then 
$$
P(\max(A,B)\ge mn)\le \sum_{k\ge m-2} 4k^4e^{-\alpha k}.
$$
It follows that 
$$
P(\max(A,B)\le mn)\to 1
$$
as $m\to\infty$.
By the same argument, if $T$ is the time  when all sites on the outer edges of $R$ reach their final state, 
$$P(T\le mn)\to 1$$ as $m\to\infty$. 

\begin{figure}[t!]
    \centering
    \includegraphics[width=0.45\linewidth]{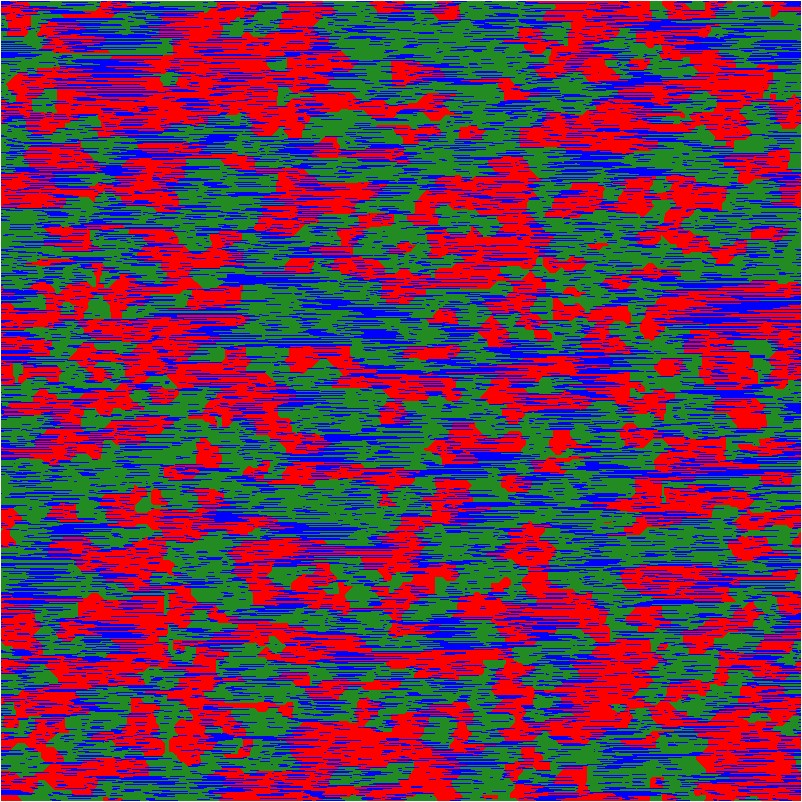}
    \caption{ Three color dynamics with one-dimensional 
    growth by blue and two-dimensional growth by red and  green, with $\cB=\{e_1\}$, $\cR=\cG=\{\pm e_1,\pm e_2\}$, 
    $p_b=0.02$, $p_r=p_g=0.001$.}
    \label{fig:3-color-221}
\end{figure}

We now remove the assumption that $q=0$, and define the configuration 
in two stages: start by deciding on the red and blue sites, then 
independently on potentially green sites with probability $q/(1-2p)$, and then finally define the green sites as those that are potentially green 
but not blue or red. It then follows by the speed-of-light argument that 
\begin{equation*}
\begin{aligned}
&\P(\text{the origin is forever not green})\ge \\
&\qquad\P(\text{there are no potentially green sites in the $2mn\times 2mn$ box around the origin})\\
&\qquad\times \P(\max(A,B)\le mn, T\le mn). 
\end{aligned}
\end{equation*}
Fix $\epsilon>0$, and choose $m$ large so that the second factor is at least $1-\epsilon$ for all sufficiently small $p$. Then for this choice of $m$, the first factor tends to $1$ as $p\to 0$. This completes the proof.
\end{proof}

\begin{open}
Assume the setting of Theorem~\ref{thm:3-color}, but now $q=cp$, for a fixed number $c$. Does the limit 
$$
\lim_{p\to 0} \P(\text{the origin is eventually green})
$$
exist?
\end{open}

We also ask what happens when red and blue grow in only one direction. The argument for Theorem~\ref{thm:3-color} breaks down in this case, since the connected components of empty sites when $p_g=0$ are no longer rectangles. See Figure~\ref{fig:3-color-211} (right). 

\begin{open}
    Assume the setting of Theorem~\ref{thm:3-color} except that $\cB = \{e_1\}$ and $\cR = \{e_2\}$. What relationship between $p$ and $q$ results in $\P(\text{the origin is eventually green})\to 1$?
\end{open}

One way to attack the first two open problems would be to try to define a
limiting continuum object to which final configurations converge when suitably scaled as 
$p\to 0$. This seems to be a substantial challenge; we end with another 
question in this vein that might be more accessible. Assume now that 
$\cB=\{e_1\}$, $\cR=\cG=\{\pm e_1,\pm e_2\}$, $p_r=p_g=p$ and $q=p^\gamma$ with $\gamma>3/2$. Observe that Theorem~\ref{thm:1vrho}
implies that the density of blue sites in the final configuration 
vanishes as $p\to 0$. Thus, the red and green sites divide the space in 
a sort of ``dynamic Voronoi tessellation'' (by contrast to the ``static'' one discussed in \cite{GG} when $q=0$). See Figure~\ref{fig:3-color-221}.

\begin{open}
Does this final configuration, scaled by $1/\sqrt{p}$, converge to 
a random two-coloring of $\bR^2$? 
\end{open}

\section*{Acknowledgments}
JG was partially supported by the Slovenian Research Agency research program P1-0285 and
Simons Foundation Award \#709425. DS was partially supported by the NSF TRIPODS grant
CCF–1740761.

\end{document}